\documentclass[11pt,A4paper]{article}

\usepackage[left=32mm,right=32mm,top=30mm,bottom=32mm]{geometry}
\usepackage{labelfig}
\usepackage{epsfig}
\usepackage{epstopdf}

\usepackage{caption}

\usepackage{xfrac}

\usepackage[percent]{overpic}

\usepackage{color}
\usepackage{amsthm,amsmath,amssymb}
\usepackage{booktabs}
\usepackage{mathpazo}
\usepackage{microtype}
\usepackage{overpic}
\usepackage{bm}
\usepackage{sectsty}
\usepackage[
	pdftitle={PDFTitle},
	pdfauthor={Hugo Parlier},
	ocgcolorlinks,
	linkcolor=linkred,
	citecolor=linkred,
	urlcolor=linkblue]
{hyperref}

\usepackage{mathtools}
\usepackage{marginnote}

\definecolor{linkred}{RGB}{199,21,133}
\definecolor{linkblue}{RGB}{16, 78, 139}

\usepackage[hang,flushmargin]{footmisc}
\usepackage{enumitem}
\usepackage{titlesec}
	\titlespacing{\section}{0pt}{12pt}{0pt}
	\titlespacing{\subsection}{0pt}{6pt}{0pt}
	
\titlelabel{\thetitle.\quad}

\makeatletter 

\long\def\@footnotetext#1{%
\H@@footnotetext{%
\ifHy@nesting 
\hyper@@anchor{\@currentHref}{#1}%
\else 
\Hy@raisedlink{\hyper@@anchor{\@currentHref}{\relax}}#1%
\fi 
}}

\def\@footnotemark{%
\leavevmode 
\ifhmode\edef\@x@sf{\the\spacefactor}\nobreak\fi 
\H@refstepcounter{Hfootnote}%
\hyper@makecurrent{Hfootnote}%
\hyper@linkstart{link}{\@currentHref}%
\@makefnmark 
\hyper@linkend 
\ifhmode\spacefactor\@x@sf\fi 
\relax 
}%

\ifFN@multiplefootnote%
\renewcommand*\@footnotemark{%
\leavevmode 
\ifhmode 
\edef\@x@sf{\the\spacefactor}%
\FN@mf@check 
\nobreak 
\fi 
\H@refstepcounter{Hfootnote}%
\hyper@makecurrent{Hfootnote}%
\hyper@linkstart{link}{\@currentHref}%
\@makefnmark 
\hyper@linkend 
\ifFN@pp@towrite 
\FN@pp@writetemp 
\FN@pp@towritefalse 
\fi 
\FN@mf@prepare 
\ifhmode\spacefactor\@x@sf\fi 
\relax%
}%
\fi 

\makeatother 

\newtheorem{thm}{Theorem}[section]

\newtheorem{cor}[thm]{Corollary}
\newtheorem{lem}[thm]{Lemma}
\newtheorem{prop}[thm]{Proposition}

\theoremstyle{definition}

\theoremstyle{remark}

\newtheorem{remark}[thm]{Remark}

\renewcommand{\phi}{\varphi}

\newcommand{\sys}{{\rm sys}}

\newcommand{\diam}{{\rm diam}}

\newcommand{\PP}{ {\mathbb P}}
\newcommand{\R}{ {\mathbb R}}

\newcommand{\arccosh}{{\rm arccosh}}

\newcommand{\be}{ \begin{equation} }
\newcommand{\ee}{ \end{equation} }

\sectionfont{\large \bfseries}
\subsectionfont{\normalsize}

\long\def\symbolfootnote[#1]#2{\begingroup%
\def\thefootnote{\fnsymbol{footnote}}\footnote[#1]{#2}\endgroup}

\def\blfootnote{\xdef\@thefnmark{}\@footnotetext}

\date{\today}

\begin{document}

{\Large \bfseries 
Systoles and diameters of hyperbolic surfaces}

{\large 
Florent Balacheff\symbolfootnote[1]{Supported by the FSE/AEI/MICINN grant RYC-2016-19334 "Local and global systolic geometry and topology" and the FEDER/AEI/MICIU grant PGC2018-095998-B-I00 "Local and global invariants in geometry". },
 Vincent Despr\'e and Hugo Parlier\symbolfootnote[7]{Despr\'e and Parlier supported by ANR/FNR project SoS, INTER/ANR/16/11554412/SoS,\\
 ANR-17-CE40-0033\vspace{.1cm} 
}
}

\vspace{0.5cm}

{\bf Abstract.} 
In this article we explore the relationship between the systole and the diameter of closed hyperbolic orientable surfaces. We show that they satisfy a certain inequality, which can be used to deduce that their ratio has a (genus dependent) upper bound.
\vspace{1.5cm}

\section{Introduction}

Geometric invariants play an important part in understanding manifolds and, when applicable, their moduli spaces. A particularly successful example has been that of systolic geometry where one studies the shortest length of a non-contractible closed curve of a non-simply connected closed Riemannian manifold (the systole).

For (closed orientable) hyperbolic surfaces of given genus $g\geq 2$, the systole becomes a function over the underlying moduli space ${\mathcal M}_g$, and has intriguing properties, such as being a topological Morse function over moduli space \cite{Akrout,BFR,SchmutzSchaller}. Observe that it is easy to construct a hyperbolic surface with arbitrarily small systole, and that a standard area argument gives an upper bound which grows like $2\log g$. Buser and Sarnak \cite{BuserSarnak} were the first to construct a family of surfaces with growing genus and with systole on the order of $\frac{4}{3} \log(g)$. Since then, there have been other constructions (see for example \cite{KSV}) but never with a greater order of growth than $\frac{4}{3} \log(g)$, and the true maximal order of growth of families of surfaces remains elusive.

In an analogous way, one can study the diameter. Here it is easy to construct surfaces with arbitrarily large diameter, but difficult to construct small diameter surfaces. Diameters are quite tricky geometric invariants, because you need to maximize among all pairs of points minimal distance, and in particular to the best of our knowledge, not a single minimal diameter surface (in its moduli space) is known. Despite these inherent difficulties, the asymptotic question was recently settled by Budzinski, Curien and Petri \cite{BCP} where they show that the minimal diameter in each genus grows asymptotically like $\log(g)$.

For any closed hyperbolic orientable surface $X$, we prove the following inequality which relates these two quantities:

\begin{thm}\label{thm:main}
The systole $\sys(X)$ and the diameter $\diam(X)$ of a closed hyperbolic orientable surface $X$ always satisfy the following inequality:
$$
4 \cosh(\sys(X)/2) \leq 3 \cosh(\diam(X)) - 1. 
$$
\end{thm}

The first to study the systole and the diameter together in the case of hyperbolic surfaces was Bavard \cite{Bavard} who proved two inequalities that if put together provide the following (see Section \ref{sec:prelim}):
$$
4 \cosh^2(\sys(X)/2) \leq 3 \cosh^2(\diam(X)) + 1.
$$
The inequality in Theorem \ref{thm:main} strengthens this relation between the two quantities. 

The relation between systole and diameter has also been studied in other contexts. In the event that a closed Riemannian manifold $M$ is non-simply connected, a cut and paste type argument rapidly shows that the ratio of systole and diameter, a natural scale invariant, satisfies:
$$
\frac{\sys(M)}{\diam(M)} \leq 2.
$$
Interestingly, equality occurs if and only if $M$ is isometric to a standard projective $\R \PP^n$ for some $n\geq 1$ \cite[Proposition 5.30]{Gromov}.

Using Theorem \ref{thm:main}, we obtain the following improvement for orientable hyperbolic closed surfaces (which is a compilation of corollaries \ref{cor:genusg} and \ref{cor:genus2} from the final section of the paper):

\begin{cor}\label{cor:alltogether}
A closed hyperbolic genus $2$ surface $X$ always satisfies
$$ \frac{\sys(X)}{\diam(X)}\leq \frac{2\,\arccosh\left(1+\sqrt{2}\right)}{\arccosh\left(\frac{5+4\sqrt{2}}{3} \right)}<\frac{8}{5},$$
and a closed hyperbolic genus $g\geq 3$ surface $X$ always satisfies the asymptotic upperbound
$$
\frac{\sys(X)}{\diam(X)} \leq 2 \left(1-\frac{\log\left(\frac{16}{\pi}\right)}{\log g}+o\left(\frac{1}{\log g}\right)\right).
$$
\end{cor}

This paper is organized as follows. Section \ref{sec:prelim} provides a short setup, before passing to the proof of Theorem \ref{thm:main} in Section \ref{sec:main}. In Section \ref{sec:final} we deduce Corollary \ref{cor:alltogether} and end the paper with some open questions.

{\bf Acknowledgment.}

We thank the referee for useful comments and remarks. 

\section{Brief setup}\label{sec:prelim}

Throughout the remainder of the paper, $X$ will be a closed orientable hyperbolic surface of genus $g\geq 2$. The moduli space $\mathcal{M}_g$ of genus $g\geq 2$ is thought of as a space of hyperbolic surfaces up to isometry. We are interested in closed geodesics on hyperbolic surfaces, and a fact that we use on a regular basis is the existence of a unique closed geodesic in a non-trivial free homotopy class of closed curve. We refer to \cite{BuserBook} for background material on hyperbolic surfaces, their curves and their moduli spaces. 

Among all closed geodesics, ones of least length are special: the systole of $X$ is the length of a shortest non-contractible curve of $X$, and is denoted $\sys(X)$. We will abusively also call a systole any non-contractible curve of $X$ whose length realizes the systole.  It is not too difficult to see, by a cut and paste argument, that in our setup, a systole is always a simple closed geodesic. The diameter of $X$, denoted $\diam(X)$, is the maximal distance between any two points of $X$, and because $X$ is closed, is a finite positive number.

Bavard \cite{Bavard} proved two inequalities that are relevant for our setup, namely that
\begin{equation}\label{eq:bavarddiam}
\cosh\left(\diam(X)\right) \geq \frac{1}{\sqrt{3} \tan \frac{\pi}{12 g - 6}}
\end{equation}
and
\begin{equation}\label{eq:bavardsys}
\cosh\left(\frac{\sys(X)}{2}\right)\leq \frac{1}{2 \sin \frac{\pi}{12 g - 6}}.
\end{equation}
Note that by a simple manipulation, one can deduce from these the inequality
$$
4 \cosh^2(\sys(X)/2) \leq 3 \cosh^2(\diam(X)) + 1.
$$

\section{Inequalities for all closed hyperbolic surfaces}\label{sec:main}

There are two competing strategies when trying to relate diameter and systole. One consists in taking diametrically opposite points and constructing a non-contractible curve of bounded length, which gives a quantifiable upper bound on the length of the systole. The other strategy is to start from a systole and use the local geometry to get a lower bound on diameter. Theorem \ref{thm:main} is the result of the latter strategy, but it is interesting to observe that the first strategy, which is slightly simpler, gives a weaker but similar result, so we very briefly outline it first.

Consider $p$ and $q$ on a surface $X$ that realize the diameter (these points may not be unique). By a variational argument, there must be at least three distance realizing paths between $p$ and $q$ (otherwise you can find points further away from each other). There are different topological configurations to consider, and using the geodesic arcs all of length the diameter, one can construct different non-contractible curves. Each length of the corresponding geodesics must be of length greater than the systole. The strategy then consists on finding an upper bound on the length of the shortest of these candidate curves. Using the fact that the total angle at both $p$ and $q$ is $2\pi$ and that $3$ arcs join $p$ and $q$, a certain pair of these geodesic arcs must intersect with the sum of their angles at most $\frac{4\pi}{3}$. The smaller the angle, the shorter the underlying curve, and then by computing the length of the curves in each of the topological situations, one obtains the following inequality (see Remark \ref{rem:othercalculus}):
$$
4 \cosh(\sys(X)/2) \leq 3 \cosh(\diam(X)) + 1.
$$

We now take an opposite strategy, beginning with observing the local geometry around a systole and then finding points that are somewhat far apart. This strategy will result in the proof of our main result. 

{\it The proof of Theorem \ref{thm:main}}

We begin by considering a systole $\sigma$ of $X$. Observe that it is a convex subset of $X$, that is a distance path between any $x,y \in \sigma$ is entirely contained in $\sigma$. Note that systoles aren't necessarily the only convex curves of a surface.

Consider two diametrically opposite points on $\sigma$ which shall be denoted $p_0$ and $q_0$. Now fix two geodesic arcs $p,q:[0,\infty)\to X$ parametrized by arc length starting respectively at $p_0$ and $q_0$ and leaving orthogonally $\sigma$ on the same side of the curve. We define two geodesic arcs $a_t$ and $b_t$ connecting $p(t)$ and $q(t)$ as the unique geodesic arc in the homotopy class obtained by following the segment of $p$ between $p(t)$ and $p(0)=p_0$, one of the two segments of $\sigma$ connecting $p_0$ and $q_0$, and the segment of $q$ between $q_0=q(0)$ and $q(t)$.
We let 
$$
T= \sup\{ t \mid \text{the only distance paths between $p(t)$ and $q(t)$ are }\, a_t \, \text{and} \, b_t \}.
$$

Note that $T>0$, otherwise there would be a third distance path between $p_0$ and $q_0$, and this would give rise to a non-contractible curve shorter than $\sigma$.

Furthermore, $T$ is finite (otherwise we will have $d(p(t),q(t))=\ell(a_t)=\ell(b_t) \to \infty$ when $t\to \infty$). And between $p(T)$ and $q(T)$ there must be (at least) a third distance path. 

We denote by $a$ and $b$ the two distance paths $a_T$ and $b_T$ between $p(T)$ and $q(T)$ and $c$ the (or a choice of a) third distance path between them. Note that $a_T$ and $b_T$ form a simple broken geodesic, which together with $\sigma$, bound a topological cylinder. By symmetry of the construction, the interior angles of the cylinder in both points are equal. This will be explicitly used in what will be called the "torus case" in the sequel. 

As they are all simple and interior disjoint, there are two topological configurations possible for $a$, $b$ and $c$: together they either form the spine of a pair of pants or of a one holed torus (see Figures \ref{fig:pants} and \ref{fig:torus}). In other words, it means that adding a small neighborhood around $a\cup b \cup c$ will result in a pair of pants or a one holed torus.
\begin{figure}[h]
\leavevmode \SetLabels
\L(.475*.78) $p_T$\\%
\L(.475*.22) $q_T$\\%
\endSetLabels
\begin{center}
\AffixLabels{\centerline{\includegraphics[width=3cm]{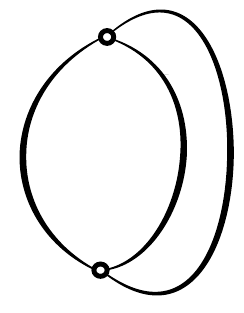}}}
\vspace{-24pt}
\end{center}
\caption{The arcs $a$, $b$ and $c$ form the spine of a pair of pants}
\label{fig:pants}
\end{figure}

\begin{figure}[h]
\leavevmode \SetLabels
\L(.425*.32) $p_T$\\%
\L(.615*.31) $q_T$\\%
\endSetLabels
\begin{center}
\AffixLabels{\centerline{\includegraphics[width=5.5cm]{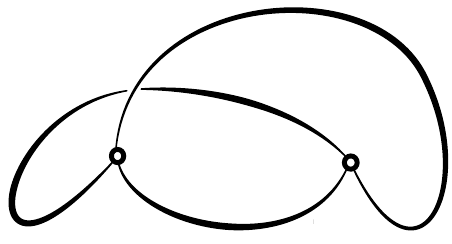}}}
\vspace{-24pt}
\end{center}
\caption{The arcs $a$, $b$ and $c$ form the spine of a one holed torus}
\label{fig:torus}
\end{figure}

The former case corresponds to having $c$ lie outside of a small embedded collar around $\sigma$. Otherwise we would be in the situation of Figure \ref{fig:badcase} where the distance path $c$ crosses the systole $\sigma$ two times. This is not possible as a cut and paste type argument rapidly shows. 

\begin{figure}[h]
\leavevmode \SetLabels
\L(.47*.77) $a$\\%
\L(.58*.425) $b$\\%
\L(.404*.31) $c$\\%
\L(.5*.06) $\sigma$\\%
\endSetLabels
\begin{center}
\AffixLabels{\centerline{\includegraphics[width=4.5cm]{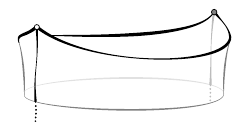}}}
\vspace{-24pt}
\end{center}
\caption{An impossible situation for $a,b,c$}
\label{fig:badcase}
\end{figure}
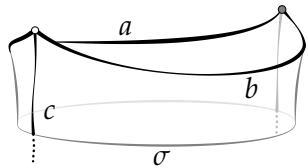

In the latter case, $c$ must cross $\sigma$ exactly one time. We will treat each of these two cases separately, and call them the pants case and the torus case. 

In both cases, the following lemma will be crucial.

\begin{lem}\label{lem:length}
Consider two interior disjoint simple arcs of length $d$ that share endpoints and form angles $0<\alpha\leq\pi$ and $0<\beta\leq\pi$. Let $L$ be the length of the closed geodesic $\gamma$ in the homotopy class of the simple closed curve formed by the two arcs. 

If $\alpha$ and $\beta$ are as in figure \ref{fig:same}

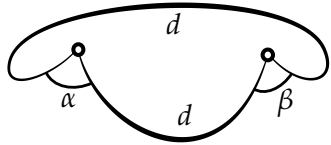
\begin{figure}[h]
\leavevmode \SetLabels
\L(.408*.315) $\alpha$\\%
\L(.597*.305) $\beta$\\%
\L(.51*.21) $d$\\%
\L(.5*.75) $d$\\%
\endSetLabels
\begin{center}
\AffixLabels{\centerline{\includegraphics[width=4.5cm]{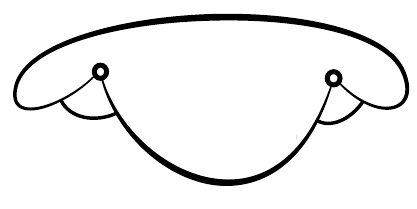}}}
\vspace{-24pt}
\end{center}
\caption{The angles $\alpha$ and $\beta$ lie on the same side of the curve}
\label{fig:same}
\end{figure}

 then 

$$
\cosh\left(\frac{L}{2}\right) = -\cos\left(\frac{\alpha}{2}\right) \cos\left(\frac{\beta}{2}\right) + \sin\left(\frac{\alpha}{2}\right) \sin\left(\frac{\beta}{2}\right) \cosh(d),
$$

and if $\alpha$ and $\beta$ are as in figure \ref{fig:opposite}

\begin{figure}[h]
\leavevmode \SetLabels
\L(.415*-.05) $\alpha$\\%
\L(.62*.44) $\beta$\\%
\endSetLabels
\begin{center}
\AffixLabels{\centerline{\includegraphics[width=5cm]{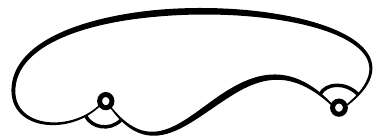}}}
\vspace{-24pt}
\end{center}
\caption{The angles $\alpha$ and $\beta$ lie on opposite sides of the curve}
\label{fig:opposite}
\end{figure}

 then 

$$
\cosh\left(\frac{L}{2}\right) = \cos\left(\frac{\alpha}{2}\right) \cos\left(\frac{\beta}{2}\right) + \sin\left(\frac{\alpha}{2}\right) \sin\left(\frac{\beta}{2}\right)\cosh(d).
$$
\end{lem}

\begin{proof}
Let $c_1$ and $c_2$ be the two simple arcs of length $d$ and let $p$ and $q$ be their endpoints (and thus intersection points). We can argue in both cases by projecting $p$ and $q$ to $\gamma$ and by then reducing the problem to hyperbolic trigonometry.

To do this, consider geodesic arcs between $p$ and $q$ and $\gamma$ that end orthogonally on $\gamma$. Cutting along these arcs gives rise to two quadrilaterals. In the situation of Figure \ref{fig:same}. the quadrilaterals are convex, whereas in the other case they are not. In both cases, the two quadrilaterals share 4 equal side lengths and 2 angles, and hence are isometric. 

\begin{figure}[h]
\leavevmode \SetLabels
\L(.346*.61) $\frac{\alpha}{2}$\\%
\L(.642*.5) $\frac{\beta}{2}$\\%
\L(.51*.34) $d$\\%
\L(.56*.14) $\frac{L}{2}$\\%
\endSetLabels
\begin{center}
\AffixLabels{\centerline{\includegraphics[width=5.5cm]{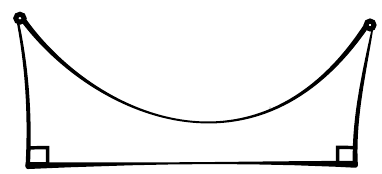}}}
\vspace{-24pt}
\end{center}
\caption{The convex quadrilateral}
\label{fig:quad2}
\end{figure}
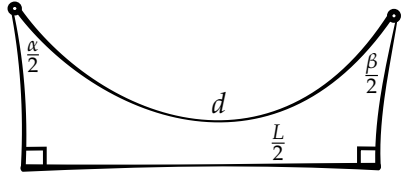

\begin{figure}[h]
\leavevmode \SetLabels
\L(.376*.75) $\frac{\alpha}{2}$\\%
\L(.61*.2) $\frac{\beta}{2}$\\%
\L(.46*.68) $d$\\%
\endSetLabels
\begin{center}
\AffixLabels{\centerline{\includegraphics[width=4.5cm]{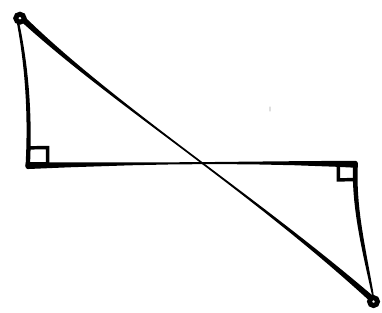}}}
\vspace{-24pt}
\end{center}
\caption{The non-convex case}
\label{fig:quad1}
\end{figure}

Thus they can be represented as in figures \ref{fig:quad2} and \ref{fig:quad1}, and in particular the angles $\alpha$ and $\beta$ are equally split between the two quadrilaterals in each case.

The formulas now follow from standard hyperbolic trigonometry (for instance both formulas can be found in \cite[Section VI.3, pp. 88-89]{Fenchel}).
\end{proof}

As a consequence, we first have the following proposition. 

\begin{prop}\label{cor:extremalsame}
Consider two interior disjoint simple arcs of length $d$ that share endpoints and form angles $0<\alpha\leq\beta< 2\pi$ such that $\alpha+\beta<2\pi$. Suppose the angles $\alpha$ and $\beta$ lie on the same side of the curve like in figure \ref{fig:same} and let $L$ be the length of the closed geodesic $\gamma$ in the homotopy class of the simple closed curve formed by the two arcs. 

Then we have
$$
\cosh\left(\frac{L}{2}\right) \leq -\cos^2\left(\frac{\theta}{2}\right) + \sin^2\left(\frac{\theta}{2}\right) \cosh(d)
$$
where $\theta= \frac{\alpha+\beta}{2}$.
\end{prop}
\begin{proof}
Observe that $\alpha<\pi$. 

If $\beta\leq \pi$, 
the upper bound follows from the first formula in Lemma \ref{lem:length} and a straightforward calculus computation that the function
$$f(x)=-\cos (x) \cos (\theta-x)+\sin (x) \sin (\theta-x) \cosh (d)$$
is strictly increasing on $[0,\theta/2]$ for $0<\theta<\pi$.

If $\beta> \pi$, we invoke the second case of the previous lemma by replacing $\beta$ with $2\pi-\beta$, and thus
\begin{eqnarray*}
\cosh\left(\frac{L}{2}\right) &=& \cos\left(\frac{\alpha}{2}\right) \cos\left(\frac{2\pi- \beta}{2}\right) + \sin\left(\frac{\alpha}{2}\right) \sin\left(\frac{2\pi -\beta}{2}\right)\cosh(d)\\
&=&-\cos\left(\frac{\alpha}{2}\right) \cos\left(\frac{\beta}{2}\right) + \sin\left(\frac{\alpha}{2}\right) \sin\left(\frac{\beta}{2}\right)\cosh(d)\\
&=&-\cos\left(\frac{\alpha}{2}\right) \cos\left(\theta -\frac{\alpha}{2}\right) + \sin\left(\frac{\alpha}{2}\right) \sin\left(\theta-\frac{\alpha}{2}\right)\cosh(d)\\
&\leq&-\cos^2\left(\frac{\theta}{2}\right) + \sin^2\left(\frac{\theta}{2}\right) \cosh(d).
\end{eqnarray*}
\end{proof}

In a similar manner we find:
\begin{prop}\label{cor:extremalopposite}
Consider two interior disjoint simple arcs of length $d$ that share endpoints and form angles $0<\alpha\leq\beta< 2\pi$ such that $\alpha+\beta<2\pi$. 
 Suppose the angles $\alpha$ and $\beta$ lie on opposite sides of the curve like in figure \ref{fig:opposite} and let $L$ be the length of the closed geodesic $\gamma$ in the homotopy class of the simple closed curve formed by the two arcs. 

Then we have
$$
\cosh\left(\frac{L}{2}\right) \leq \cos^2\left(\frac{\theta}{2}\right) + \sin^2\left(\frac{\theta}{2}\right) \cosh(d)
$$
where $\theta= \frac{\alpha+\beta}{2}$.
\end{prop}

Using these two propositions we can now finish the proof of Theorem \ref{thm:main}.

{\it Pants case}

Any pair of the three arcs $a,b,c$ forms a potential homotopy class of curve. Each pair also forms angles at the two endpoints. 

\begin{figure}[h]
\leavevmode \SetLabels
\L(.47*.49) $a$\\%
\L(.591*.24) $b$\\%
\L(.5*.85) $c$\\%
\endSetLabels
\begin{center}
\AffixLabels{\centerline{\includegraphics[width=4.5cm]{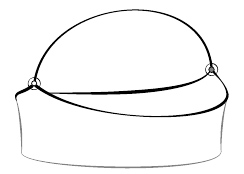}}}
\vspace{-24pt}
\end{center}
\caption{The three arcs $a,b,c$}
\label{fig:goodcase}
\end{figure}

In particular this means we are in the situation of figure \ref{fig:goodcase}. The sum of all 6 angles is $4\pi$, hence there is a pair of arcs among $a,b,c$ with the sum of the corresponding angles less or equal to $\frac{4}{3}\pi$. If we denote by $L$ the length of the corresponding closed geodesic, and appealing to Proposition \ref{cor:extremalsame}, we have the following inequality:

\begin{equation}
\cosh\left(\frac{L}{2}\right) \leq -\cos^2\left(\frac{\pi}{3}\right) + \sin^2\left(\frac{\pi}{3}\right) \cosh(d)= -\frac{1}{4} + \frac{3}{4} \cosh(d). 
\end{equation}

As $L \geq \sys(X)$ and $d \leq \diam(X)$, this proves the desired inequality in the pants case:

\begin{equation}\label{eq:pants}
\cosh\left(\frac{\sys(X)}{2}\right) \leq - \frac{1}{4} + \frac{3}{4} \cosh(\diam(X)). 
\end{equation}

{\it Torus case}

In this case, we proceed in a similar manner to establish an inequality that will turn out to be strictly stronger for the distances we're interested in. 

We are in the situation of Figure \ref{fig:case2} with angles $\alpha$ and $\beta$ whose sum is $2 \pi$. Here we decisively take advantage of the fact that the angles between $a$ and $b$ at each endpoint coincide.

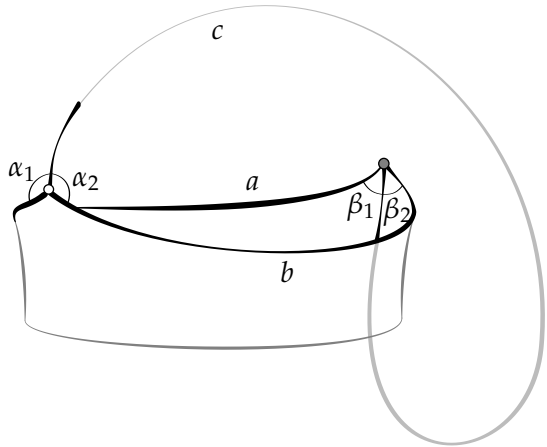
\begin{figure}[h]
\leavevmode \SetLabels
\L(.262*.6) $\alpha_1$\\%
\L(.318*.58) $\alpha_2$\\%
\L(.558*.52) $\beta_1$\\%
\L(.59*.51) $\beta_2$\\%
\L(.47*.57) $a$\\%
\L(.5*.37) $b$\\%
\L(.44*.9) $c$\\%
\endSetLabels
\begin{center}
\AffixLabels{\centerline{\includegraphics[width=7.5cm]{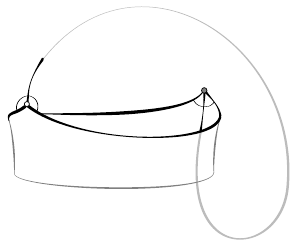}}}
\vspace{-24pt}
\end{center}
\caption{The angles are complementary}
\label{fig:case2}
\end{figure}

The arc $c$ splits each angle into two parts (denoted $\alpha_1,\alpha_2$ and $\beta_1,\beta_2$ as in the figure). By the same argument as above, there is a pair of angles $\alpha_i,\beta_i$ with $\alpha_i + \beta_i \leq \pi$. The choice of angles implies a choice of either $a$ or $b$. We denote again by $L$ the length of the corresponding closed geodesic. By applying Proposition \ref{cor:extremalopposite}, we obtain

\begin{equation}
\cosh\left(\frac{L}{2}\right) \leq \cos^2\left(\frac{\pi}{4}\right) + \sin^2\left(\frac{\pi}{4}\right) \cosh(d)= \frac{1}{2} + \frac{1}{2} \cosh(d). 
\end{equation}

Now $L\geq \sys(X)$ and $d \leq \diam(X)$, hence 

\begin{equation}\label{eq:torus}
\cosh\left(\frac{\sys(X)}{2}\right) \leq \frac{1}{2} + \frac{1}{2} \cosh(\diam(X)).
\end{equation}

Clearly for large enough values of the diameter, the above inequality is stronger than the desired one from Theorem \ref{thm:main}. We now show that in our situation this will always be the case. By contradiction, suppose that Inequality \ref{eq:torus} holds, but not Inequality \ref{eq:pants}. We would have
$$
\frac{1}{2} \cosh(\diam(X)) + \frac{1}{2} > \frac{3}{4} \cosh(\diam(X)) - \frac{1}{4} 
$$
and hence
$$
\cosh(\diam(X)) <3.
$$

We appeal to Bavard's diameter bounds stated in Section \ref{sec:prelim}. A consequence of Inequality \ref{eq:bavarddiam} is a universal lower bound on the diameter of any closed orientable hyperbolic surface given by his lower bound for genus $2$, namely
$$
\cosh(\diam(X))\geq \frac{1}{\sqrt{3} \tan \frac{\pi}{18}}= 3.27\ldots>3.
$$
for any closed genus $g\geq 2$ surface. Thus in the torus case, Inequality \ref{eq:pants} holds as well, which completes the proof of Theorem \ref{thm:main}.

\begin{remark}\label{rem:othercalculus}
We quickly explain how to deduce the weaker inequality 
$$
\cosh\left(\frac{\sys(X)}{2}\right) \leq \frac{1}{4} + \frac{3}{4} \cosh(\diam(X)). 
$$
using the first strategy. 

Consider $p$ and $q$ on the surface $X$ that realize the diameter (these points may not be unique). By a variational argument, there must be at least three distance realizing paths between $p$ and $q$ that we denote by $a,b,c$. We have, as before, two topological cases corresponding to figures \ref{fig:pants} and \ref{fig:torus}, and there is a pair of arcs among $a,b,c$ with the sum of the corresponding angles less or equal to $\frac{4}{3}\pi$.

If we are in the pants case, we find that $\cosh\left(\frac{\sys(X)}{2}\right) \leq - \frac{1}{4} + \frac{3}{4} \cosh(\diam(X))$ using Proposition \ref{cor:extremalsame}.

If we are in the torus case, we only find that $\cosh\left(\frac{\sys(X)}{2}\right) \leq \frac{1}{4} + \frac{3}{4} \cosh(\diam(X))$ using Proposition \ref{cor:extremalopposite}.
\end{remark}

\section{The ratio between systole length and diameter}\label{sec:final}

The main goal of this section is to use Theorem \ref{thm:main} to give bounds on the ratio between systole and diameter. We begin by a general inequality that depends on the genus.

\begin{cor}\label{cor:genusg}
For any $X\in {\mathcal M}_g$, the ratio of systole and diameter satisfies
$$
\frac{\sys(X)}{\diam(X)}\leq \frac{2\,\arccosh\left(\frac{1}{2\sin\left(\frac{\pi}{12g-6}\right)}\right)}{\arccosh\left(\frac{2}{3\sin\left(\frac{\pi}{12g-6}\right)}+\frac{1}{3}\right)}= 2 \left(1-\frac{\log\frac{16}{\pi}}{\log g}+o\left(\frac{1}{\log g}\right)\right).
$$
\end{cor}

\begin{proof}
By inequality \ref{eq:pants}
$$
\diam(X) \geq \arccosh \left(\frac{4}{3} \cosh\left(\frac{\sys(X)}{2}\right) + \frac{1}{3}\right) 
$$
and thus
$$
\frac{\sys(X)}{\diam(X)} \leq \frac{\sys(X)}{\arccosh \left(\frac{4}{3} \cosh\left(\frac{\sys(X)}{2}\right) + \frac{1}{3}\right)}.
$$
Observe that the function
$$s\mapsto \frac{s}{\arccosh \left(\frac{4}{3} \cosh\left(\frac{s}{2}\right) + \frac{1}{3}\right)}$$ is increasing.
We now use the upper bound on the length of the systole (Inequality \ref{eq:bavardsys}) to conclude.
\end{proof}

We point out that working with the inequality
$$
\frac{\sys(X)}{\diam(X)} \leq \frac{2\, \arccosh \left(\frac{3}{4}\cosh (\diam(X))-\frac{1}{4}\right)}{\diam(X)}
$$
and the lower bound on the diameter (Inequality \ref{eq:bavarddiam})) leads to a weaker estimate.\\

In genus $2$, there is an optimal upper bound on systole length \cite{Jenni,SchmutzSchaller}. Hence the above strategy gives the following stronger result:
\begin{cor}\label{cor:genus2}
For any closed hyperbolic $X$ of genus $2$ we have
$$ \frac{\sys(X)}{\diam(X)}\leq \frac{2\,\arccosh\left(1+\sqrt{2}\right)}{\arccosh\left(\frac{5+4\sqrt{2}}{3} \right)}<\frac{8}{5}.$$
\end{cor}
\begin{proof}
We apply the strategy above but this time using the (sharp) inequality on systole length due to Jenni \cite{Jenni} that states that 
$$
\sys(X) \leq 2\, \arccosh(1+\sqrt{2}). 
$$
\end{proof}

{\it Further queries.} These results raise the question on the relationship between diameters and systoles, and in particular on how their ratio behaves with respect to genus. 

Similarly to the systole function, for any given moduli space $\mathcal{M}_g$, the ratio $\frac{\sys(X)}{\diam(X)}$ reaches a maximum. What properties might these surfaces have? Can one find a surface that reaches the maximum ratio in its moduli space?

The relevant asymptotic question is to study the following quantity $R$:
$$
R:=\limsup_{g \to \infty}\left( \max_{X\in {\mathcal M}_g}\frac{ \sys(X)}{\diam(X)}\right).
$$
What is the value of $R$?

The Buser-Sarnak construction (see also \cite{KSV}) gives surfaces with large systole on the order $\frac{4}{3} \log g$ but also with expander like properties, which in particular means that the diameters also grow roughly on the order of $\log(g)$ (but without any explicit bounds). Hence $R$ is a real strictly positive number, less or equal to $2$. As for the systole growth, there is no obvious reason why one couldn't replace the $\limsup$ with a $\lim$, but as far as we know, there are no behavioral results that would allow one to do this. Finally note that Corollary \ref{cor:genusg} above doesn't bring anything new, because our upper bounds for the ratio limit to the obvious upper bound of $2$.


{\it Addresses:}\\
Department of Mathematics, Universitat Aut\`onoma de Barcelona, Spain\\
Polytech Nancy, Universit\'e de Lorraine, LORIA\\
Department of Mathematics, University of Luxembourg, Esch-sur-Alzette, Luxembourg

{\it Emails:}\\
fbalacheff@mat.uab.cat\\
vincent.despre@inria.fr\\
hugo.parlier@uni.lu


\begin{thebibliography}{99}

\bibitem{Akrout} Akrout, Hugo. Singularit\'es topologiques des systoles g\'en\'eralis\'ees. Topology 42 (2003), no. 2, 291--308.

\bibitem{Bavard} Bavard, Christophe. Disques extr\'emaux et surfaces modulaires. Ann. Fac. Sci. Toulouse Math. (6), 5 no. 2 (1996) 191--202.


\bibitem{BuserBook} Buser, Peter. Geometry and spectra of compact Riemann surfaces. Progress in Mathematics, 106. Birkh\"auser Boston, Inc., Boston, MA, 1992.

\bibitem{BuserSarnak} Buser, Peter and Sarnak, Peter. On the period matrix of a Riemann surface of large genus. Invent. Math. 117 (1994), no. 1, 27--56.

\bibitem{BFR} Bourque, Maxime Fortier and Rafi, Kasra. Local maxima of the systole function. J. Eur. Math. Soc. (JEMS), to appear. 

\bibitem{BCP} Budzinski, Thomas, Curien, Nicolas and Petri, Bram. On the minimal diameter of closed hyperbolic surfaces. Duke Math. J., to appear.

\bibitem{Fenchel} Fenchel, Werner. Elementary geometry in hyperbolic space. With an editorial by Heinz Bauer. De Gruyter Studies in Mathematics, 11. Walter de Gruyter \& Co., Berlin, 1989.

\bibitem{Gromov} Gromov, Misha. Metric structures for Riemannian and non-Riemannian spaces. Based on the 1981 French original. With appendices by M. Katz, P. Pansu and S. Semmes. Translated from the French by Sean Michael Bates. Reprint of the 2001 English edition. Modern Birkh\"auser Classics. Birkh\"auser Boston, Inc., Boston, MA, 2007.

\bibitem{Jenni} Jenni, Felix.
\"Uber den ersten Eigenwert des Laplace-Operators auf ausgew\"ahlten Beispielen kompakter Riemannscher Fl\"achen [On the first eigenvalue of the Laplace operator on selected examples of compact Riemann surfaces]. Comment. Math. Helv. 59 (1984), no. 2, 193--203. 

\bibitem{KSV} Katz, Mikhail G., Schaps, Mary and Vishne, Uzi. Logarithmic growth of systole of arithmetic Riemann surfaces along congruence subgroups. J. Differential Geom. 76 (2007), no. 3, 399--422.

\bibitem{SchmutzSchaller} Schmutz Schaller, Paul. Geometry of Riemann surfaces based on closed geodesics. Bull. Amer. Math. Soc. (N.S.) 35 (1998), no. 3, 193--214.




\end{thebibliography}
\end{document}